\documentclass[11pt]{amsart}
\usepackage[shortalphabetic]{amsrefs}
\usepackage[shortalphabetic]{amsrefs}
\usepackage[utf8]{inputenc}
\usepackage{amsthm}
\usepackage[utf8]{inputenc}
\usepackage{appendix}
\usepackage[top= 2.5cm, bottom=2.5cm, left=2.5cm, right=2.5cm]{geometry}

\usepackage{cite}
\usepackage{amssymb,amsmath,amsthm}
\usepackage{hyperref}
\newcommand{\amsprimary}[1]{{\footnotesize\noindent AMS 2020 \textit{Mathematics subject
classification:} Primary #1\vspace{1pc}}}
\newcommand{\keywordsnames}[1]{{\footnotesize\noindent\textit{Key words:} #1\vspace{1pc}}}

\newtheorem{theorem}{Theorem}

\newtheorem{teo}{Theorem}

\newtheorem{corollary}[teo]{Corollary}
\newtheorem{lemma}[teo]{Lemma}

\theoremstyle{definition}

\title[]{A Polyharmonic Liouville Hierarchy on 
Complete Manifolds of Nonnegative Ricci Curvature}
\author{John E. Bravo and Jean C. Cortissoz }
\email{jcortiss@uniandes.edu.co, j.bravob@uniandes.edu.co}
\address{Department of Mathematics, Universidad de los Andes, Bogot\'a DC, Colombia}
\date{}

\begin{document}

%In this paper we explore Liouville's theorem on Riemannian cones as defined below.
%We also study the Strong Liouville Property, that is, the property of a
% cone having spaces of harmonic functions of a fixed
%polynomial growth of finite dimension.
%We explore Liouville’s theorem and the Strong Liouville Property (SLP) for harmonic functions on Riemannian cones and surfaces. Our approach reframes the classical Liouville property in terms of the growth of radial eigenfunctions 
%(in the case of manifolds with
%rotational symmetry), allowing us to recover and sharpen known results under minimal assumptions. We provide explicit estimates for the slowest-growing non-constant harmonic functions on cones and surfaces, and construct examples where doubling fails but Liouville and SLP still hold. Finally, we extend the analysis to 
%$p$-subharmonic functions, proving a nonlinear Liouville theorem %under curvature conditions resembling Milnor's celebrated criterion
%for harmonic functions.
\begin{abstract}
In this paper, we establish a complete Liouville--type hierarchy for polyharmonic functions on Riemannian manifolds with nonnegative Ricci curvature. Extending Yau's classical result for harmonic functions and our recent biharmonic Liouville theorem, we prove that on any complete manifold of nonnegative Ricci curvature, every $k$--polyharmonic function of growth $o(r^{2(k-1)})$ must in fact be $(k-1)$--polyharmonic. Iterating this procedure yields the result that all polyharmonic functions of sublinear growth are constant. 
The key innovation is a new $L^{2}$ estimate for the Laplacian of a polyharmonic function, obtained by induction through a delicate cutoff construction combined with a hole--filling argument. This provides the first sharp geometric extension of the Euclidean classification of polyharmonic functions to manifolds of nonnegative Ricci curvature, and completes a natural hierarchy of Yau--type Liouville theorems for iterates of the Laplacian.
\end{abstract}

\maketitle

\keywordsnames{Liouville theorem; polyharmonic functions; polynomial growth}

{\amsprimary {31C05, 53C21}}

\tableofcontents

%=========================================================
%  Caccioppoli inequality for the Laplacian of a biharmonic
%  function (local version with dependence on R-r and Ricci bound)
%=========================================================

% ======= Minimal preamble suggestions =======
% \usepackage{amsmath, amssymb, amsthm}
% \numberwithin{equation}{section}
% \newtheorem{prop}{Proposition}[section]
% \newtheorem{rem}[prop]{Remark}
% ============================================
\section{Introduction}

Harmonic functions have long been a central object of study in geometric analysis. Among the most celebrated and foundational results in the field is Yau’s Liouville theorem \cite{Yau75}, which asserts that on a complete Riemannian manifold with nonnegative Ricci curvature, any bounded harmonic function must be constant. This theorem not only captures a fundamental rigidity phenomenon, but also inaugurated a far-reaching program aimed at understanding how analytic growth conditions reflect the large-scale geometry of the underlying manifold.

\medskip
Since Yau’s work, the study of harmonic functions and spaces of harmonic functions has developed extensively. Notable contributions include the work of Li and Tam \cite{LiTam1989}, as well as Colding and Minicozzi \cite{ColdingMinicozzi97, ColdingMinicozziAnnals}, who investigated the structure and finite dimensional properties of spaces of harmonic functions of polynomial growth. Further important developments include results of Ni and Tam \cite{NiTam2003} on pluriharmonic functions, and more recently Liu’s work on holomorphic functions via three--circle theorems \cite{LiuThreeCircle}. Collectively, these results demonstrate that under nonnegative Ricci curvature, harmonic functions exhibit rigidity properties closely resembling those of the Euclidean setting.

\medskip
By contrast, Liouville-type phenomena for higher-order elliptic equations have received comparatively little attention. This gap is striking given the intrinsic analytic and geometric relevance of higher-order operators, such as the Paneitz operator in conformal geometry or the operators arising in the study of polyharmonic flows for curves. Aside from a few notable exceptions, there are very few results addressing Liouville-type properties in this broader context. Among these, we mention the works of Branding \cite{Branding18, Branding2021}, where Liouville-type theorems are obtained for biharmonic \emph{maps} under smallness assumptions on certain Sobolev norms. 

\medskip
Recently, the authors proved the following result (see \cite{BravoCortissoz2025}, as well as the related work \cite{WangZhu2025}).

\begin{theorem}
\label{thm:biharmonic}
Let $(M,g)$ be a complete Riemannian manifold with a pole and such that 
$\mbox{Ric}_g\geq 0$. 
Then any biharmonic function $u\in C^{4}(M)$ of subquadratic growth is harmonic. 
In particular, any biharmonic function of sublinear growth must be constant.
\end{theorem}

The purpose of the present paper is to extend Theorem~\ref{thm:biharmonic} to the general polyharmonic setting. Recall that a $k$-polyharmonic function on a Riemannian manifold $(M,g)$ is a function $u\in C^{2k}(M)$ satisfying
\[
\Delta^k u = 0,
\]
where $\Delta$ denotes the Laplacian.

\medskip
Our main result is the following.

\begin{theorem}
\label{thm:main}
Let $(M,g)$ be a complete Riemannian manifold with 
$\mbox{Ric}_g\geq 0$. 
Then any $k$-polyharmonic function $u\in C^{2k}(M)$ satisfying
\[
u=o(r^{2(k-1)}) \quad \mbox{as}\quad r\rightarrow \infty,
\]
where $r=d(p,x)$, is $(k-1)$-polyharmonic. 
In particular, any polyharmonic function of sublinear growth must be constant.
\end{theorem}

\medskip
As is well known, Theorem~\ref{thm:main} has long been established in the Euclidean setting. In $\mathbb{R}^n$, the critical growth rate that separates genuinely $k$-polyharmonic behavior from lower-order behavior is precisely $r^{2(k-1)}$, as follows from Almansi’s decomposition. This shows that Theorem~\ref{thm:main} is sharp and constitutes a natural geometric extension of the classical Euclidean theory.

\medskip
We now briefly describe the main analytic difficulty underlying the proof of Theorem~\ref{thm:main}. We begin by recalling the strategy used in the proof of Theorem~\ref{thm:biharmonic}. In \cite{BravoCortissoz2025}, the argument relied on an energy estimate for the $L^2$-norm of $\Delta u$ together with a Caccioppoli-type inequality,
\begin{equation}
\label{eq:Cacciopoli}
\int_{B_s}|\nabla u|^2\,dV \leq \frac{C}{(r-s)^2}\int_{B_r}u^2\,dV,
\end{equation}
valid for both harmonic and biharmonic functions. This estimate was then used to derive
\begin{equation}
\label{eq:Cacciopoli_Laplacian}
\int_{B_s}|\Delta u|^2\,dV \leq \frac{C}{(r-s)^4}\int_{B_r}u^2\,dV,
\end{equation}
which, combined with a mean value inequality, yields the desired rigidity.

\medskip
At first glance, these arguments suggest the possibility of an inductive approach to Theorem~\ref{thm:main}. However, the proof of \eqref{eq:Cacciopoli} is fundamentally based on the Bochner formula, and this approach does not extend in any straightforward way to polyharmonic functions beyond the biharmonic case. To overcome this obstacle, we establish a direct analogue of \eqref{eq:Cacciopoli_Laplacian} for \emph{arbitrary polyharmonic functions}. This new estimate is the most delicate part of the argument and forms the analytic backbone of the inductive procedure. Although the biharmonic case serves as the base step, the proof requires a genuinely new estimate that may be of independent interest.

\medskip
Taken together with Yau’s theorem and our previous work on biharmonic functions, the results of this paper complete, on manifolds with nonnegative Ricci curvature, a natural hierarchy of Yau--Liouville type theorems for the Laplacian and its iterates. They also raise the question of whether similar rigidity phenomena persist for more general higher-order elliptic operators, as they do in the Euclidean setting.

\medskip
Beyond its analytic content, Theorem~\ref{thm:main} may be interpreted as a rigidity principle for complete manifolds with nonnegative Ricci curvature. It shows that, at the level of critical growth, such manifolds are analytically indistinguishable from Euclidean space with respect to iterates of the Laplacian. In particular, the behavior of polyharmonic function in
nonnegative Ricci curvature seems to mirror Almansi’s decomposition in $\mathbb{R}^n$. %From this perspective, the polyharmonic Liouville hierarchy provides a new analytic invariant of the large--scale geometry of the manifold, encoding curvature constraints through the structure of function spaces rather than through pointwise geometric data.

\subsection{}

The interplay between the behavior of polyharmonic functions and curvature remains largely unexplored and may reveal further rigidity phenomena. As shown in this paper, when $\mbox{Ric}_g\geq 0$, the qualitative behavior closely parallels that of Euclidean space. In contrast, in negatively curved settings the situation is markedly different. In particular, the authors showed in \cite{BravoCortissoz} that there exist curvature regimes in which every bounded biharmonic function must be harmonic, while nonconstant bounded harmonic functions still exist. This naturally leads to the question of how $k$-polyharmonic functions behave in such settings.

\subsection{}

In a previous version of this manuscript, we imposed one of the following additional assumptions:
\begin{itemize}
\item[(H1)] the geodesic spheres centered at $p$ are mean-convex with respect to the radial vector field, or
\item[(H2)] the square of the distance function from $p$, denoted by $d$, is convex.
\end{itemize}
Either assumption yields a lower bound for the Laplacian of the distance function, which is needed to construct suitable cutoff functions. However, we later observed that the condition $\mbox{Ric}_g\geq 0$, together with the existence of a pole, already implies (H1), rendering it redundant (see \cite{BravoCortissoz2025}). Subsequently, inspired by the recent work \cite{WangZhu2025} and with the generous assistance of Professor Wang, we adopted the cutoff technology developed in \cite{BianchiSetti2018}, which requires only the assumption $\mbox{Ric}_g\geq 0$. We are grateful to Professors Wang and Zhu for their help.

\subsection{Organization of the paper}
The material of this paper is organized as follows. In Section
\ref{section:technical}, we collect some facts about cutoff functions
and the hole-filling technique needed in the proof of the main
$L^2$-estimates needed in the proof of our main result. These
$L^2$ estimates are presented and proved in Section \ref{sect:L2estimates}.
Finally, Theorem \ref{thm:main} is proved in Section \ref{sect:proof_main_thm}.

\section{Some Technical Lemmas}
\label{section:technical}

In this section, we collect some technical lemmas that will be used in the proof of our main result.

\medskip
The following pair of lemmas are proved in \cite{BravoCortissoz2025} under the hypothesis that
the manifold has a pole. However, using the cutoff technology from \cite{BianchiSetti2018}, we can give a proof of
its validity without this extra assumption.

\begin{lemma}[Cutoff with explicit bounds]\label{lem:cutoff}
Let $(M^n,g)$ be a complete Riemannian manifold 
with $\mbox{Ric}_g\geq 0$
Set $B_\rho:=B_\rho(p)$. Let $r$ be such that and for $0<\frac{1}{8}R\leq r\leq R$.
Then there exists $\chi\in C_c^\infty(B_R)$ such that
\[
0\le \chi\le 1,\qquad
\chi\equiv 1\ \text{on } B_r,\qquad
\mbox{supp}\left(\chi\right)\subset B_R,
\]
and 
\[
\left|\nabla \chi\right|\le \frac{C(n)}{R-r},
\quad
|\Delta\chi|\le \frac{C(n)}{(R-r)^{2}}
\quad \text{on}\quad B_R.
\]
\end{lemma}
\begin{proof}
    Take $\chi=\phi$, where $\phi$ is as in Corollary 2.3 in 
    \cite{BianchiSetti2018}.  Notice that given $r$, we can write $R=\gamma r$ with $\gamma>1$, and in this case,
    we obtain the following estimates for the gradient and Laplacian of the cutoff function.
    \[
    \left|\nabla \chi\right|\leq \frac{C_1}{R}=\frac{\left(\gamma-1\right)C_1}{\left(\gamma-1\right) R}\leq
    \frac{7C_1}{R-r},
    \]
    where we used that $r\geq \frac{1}{8}R$.

\medskip
    Regarding the Laplacian, as we are in the case $\alpha=2$, we have again that 
    \[
    \left|\Delta \chi\right|\leq \frac{C_2}{R^2}\leq \frac{\left(\gamma-1\right)^2C_2}{\left(\gamma-1\right)^2 R^2}\leq
    \frac{49C_2}{\left(R-r\right)^2},
    \]
    and the lemma is proved.
\end{proof}

As an immediate consequence we have (see \cite{BravoCortissoz2025}).
\begin{lemma}
\label{corollary:special_cutoff}
Let $(M^n,g)$ be a complete Riemannian manifold.  Set $B_\rho:=B_\rho(p)$, assume $\mbox{Ric}_g\ge 0$ on $B_R$,
 and let $0<r<R$.
Then there exists $\varphi\in C_c^\infty(B_r)$ such that $0\leq \varphi\leq 1$,
$\varphi=1$ on $B_s$, and for $\dfrac{1}{8}r<s<r<R$,
    \[
    \left|\nabla \varphi\right|^2\leq \frac{C}{\left(r-s\right)^2}\varphi\quad
    \text{ and } \quad
    \frac{\left|\Delta\left(\varphi^4\right)\right|^2}{\varphi^4}\leq \frac{C}{\left(r-s\right)^4}.
    \]
\end{lemma}
In particular, we may assume that the weaker bounds
\[|\Delta \varphi| \le \frac{C}{(r-s)^2} \quad \mbox{and}\quad
 \frac{|\Delta(\varphi^2)|}{\varphi} \le \frac{C}{(r-s)^2} \]
hold and will be used below.

\medskip
The following lemma and its proof can also be found in \cite{BravoCortissoz2025}.
\begin{lemma}[Hole-filling iteration]
\label{lemma:hole-filling}
Let $F,G:[0,R]\to[0,\infty)$ be functions, $\alpha >0$ and assume that there exists 
$\theta\in[0,1)$ such that for all $0<s<r\leq R$ one has
\begin{equation}\label{ineq:hole-filling}
    F(s) \;\leq\; \theta\,F(r) \;+\; \frac{1}{(r-s)^{\alpha}}\,G(r).
\end{equation}
Suppose further that $G$ is bounded above by $M$. Then for any $\lambda\in(\theta^{1/\alpha},1)$,
\begin{equation}\label{ineq:iteration}
    F(s) \;\leq\; 
    \frac{(1-\lambda)^{-\alpha}}{1-\theta\lambda^{-\alpha}}\,
    \frac{M}{(r-s)^{\alpha}}, 
    \qquad 0<s<r\leq R.
\end{equation}
\end{lemma}

The following basic estimate is Theorem 6.2
in Chapter II from \cite{SchoenYau1994}.

\begin{lemma}[Mean value inequality for $u^2$]
\label{lemma:mean_value_ineq}
Let $(M^n,g)$ be a complete Riemannian manifold and fix $x_0\in M$.
Assume a lower Ricci bound $\mbox{Ric}_g \ge -(n-1)K g$ on $B_{2a}:=B_{2a}(x_0)$ for some $K\ge0$.
If $u\ge0$ is (weakly) subharmonic on $B_{2a}$ (i.e.\ $\Delta u\ge0$ in $B_{2a}$), then
\begin{equation}\label{eq:MV-u2}
\sup_{B_a} u^{2}
\;\le\;
C(n)\,\exp\!\big(C(n)\sqrt{K}\,R\big)\,
\frac{1}{\mathrm{Vol}(B_{2a})}\int_{B_{2a}} u^{2}.
\end{equation}
In particular, when $K=0$ (nonnegative Ricci curvature on $B_{2a}$),
\begin{equation}\label{eq:MV-u2-Ric0}
\sup_{B_a} u^{2}
\;\le\;
\frac{C(n)}{\mathrm{Vol}(B_{2a})}\int_{B_{2a}} u^{2}.
\end{equation}
\end{lemma}

\subsection{An estimate for the Laplacian of biharmonic functions}
\label{sect:Laplacian estimate biharmonic}
The following estimate is the basis of
our inductive argument in order to prove \eqref{eq:Cacciopoli_Laplacian}. 
As discussed in the appendix, it is a consequence of the arguments given in \cite{BravoCortissoz2025}.

\begin{lemma}\label{cor:L2-Estimate Laplacian Biharmonic}
Let $(M,g)$ be an $n$–dimensional complete Riemannian manifold with $\mbox{Ric}_g\ge 0$.
Let $a>0$ and let $u\in C^\infty(B_{a})$ be Biharmonic, i.e.\ $\Delta^2 u=0$ on $B_{a}$.
Then,  there is a constant $C>0$ such that for every $s,r\in(0,a)$, with $0<\frac{1}{8}r<s < r<a$ 
\begin{equation}\label{eq:L2-HF-polh}
\int_{B_s}|\Delta u|^2 dV \leq \frac{C}{(r-s)^4}\int_{B_{r}} u^2 dV.
\end{equation}
    
\end{lemma}

\section{Main estimates and proof of Theorem \ref{thm:main}}

\subsection{$L^2$ estimates}
\label{sect:L2estimates}
The following energy estimate is valid for any sufficiently smooth function.

\begin{lemma}\label{lem:L2-Estimate_Gradient_N}
Let $(M,g)$ be an $n$–dimensional complete Riemannian manifold with $\mathrm{Ric}_g\ge 0$.
Let  $r,s,a>0$ be such that $s<r<a$,
and let $u\in C^\infty(B_{a})$, and let $\varphi$ be a cutoff function as in Lemma
\ref{corollary:special_cutoff}.
Then,  there exists a constant
$C=C(n)>0$ such that
\begin{equation}\label{eq:L2-HF-k pol}
\int_{B_r}  |\nabla \Delta u|^2 \varphi^2 dV \leq \frac{(r-s)^2}{C} \int_{B_{r}} |\Delta^2 u|^2\varphi^2  dV + \frac{C}{(r-s)^2}\int_{B_{r}} |\Delta u|^2 dV.
\end{equation}

\end{lemma}

\begin{proof}

Fix $0<s< r <a$. Let $\varphi\in C_c^\infty(B_{r}(p))$ as in 
Lemma \ref{corollary:special_cutoff}, that is
\[
0\le \varphi\le 1,\qquad
\varphi\equiv 1 \;\text{ on } B_s(p),\qquad
\mathrm{supp}\,\varphi \subset B_{r}(p).
\]
with
\begin{equation}\label{eq:cutoff-grad-lap}
|\nabla\varphi|\le \frac{C}{r-s},
\qquad
|\Delta\varphi|\le \frac{C}{(r-s)^2}.
\end{equation}
We shall use the notation $B_{\rho}=B_{\rho}\left(p\right)$.

\medskip
Let $\Delta u = f$, by integration by parts on $B_{r}$,
\begin{align}
\int_{B_{r}} \varphi^2\,|\nabla  f|^2\,dV
&= -\int_{B_{r}} \varphi^2\,f\,\Delta f\,dV
   - 2\int_{B_{r}} \varphi\,f\,\langle\nabla\varphi,\nabla f\rangle\,dV.
\label{eq:1st-energy-identity}
\end{align}
Applying Young’s inequality to the last term, for any $\varepsilon>0$,
\[
\left|2\int_{B_{r}} \varphi\,f\,\langle\nabla\varphi,\nabla f\rangle\,dV\right|
\le \varepsilon \int_{B_{r}} \varphi^2 |\nabla f|^2\,dV
   + \frac{C}{\varepsilon(r-s)^2}\int_{B_{r}} f^2\,dV,
\]
where we used $|\nabla\varphi|\le C/(r-s)$ from \eqref{eq:cutoff-grad-lap}.
Choosing $\varepsilon=\tfrac12$ and using \eqref{eq:1st-energy-identity},
and Young's inequality again with $\delta>0$,
\begin{align*}
\int_{B_{r}} \varphi^2|\nabla f|^2\,dV
&\leq
2\int_{B_{r}} \varphi^2|f\,\Delta f|\,dV
+\frac{C}{(r-s)^2}\int_{B_{r}} f^2\,dV,\\
&\leq \delta \int_{B_{r}}f^2 \varphi^2 dV+\frac{1}{\delta} \int_{B_{r}} |\Delta f|^2 \varphi^2 dV+\frac{C}{(r-s)^2}\int_{B_{r}} f^2\,dV,
\end{align*}
or
\begin{equation*}
    \int_{B_{r}} \varphi^2 |\nabla \Delta u|^2 dV \leq \delta \int_{B_{r}} |\Delta u|^2 \varphi^2 dV +\frac{1}{\delta} \int_{B_{r}} |\Delta^2 u|^2 \varphi^2 dV + \frac{C}{(r-s)^2}\int_{B_{r}} |\Delta u|^2 dV.
\end{equation*}

Taking $\delta=\frac{C}{(r-s)^2}$ we obtain,

\begin{equation*}
\label{lem:grad}
    \int_{B_{r}} \varphi^2 |\nabla \Delta u|^2 dV \leq \frac{(r-s)^2}{C} \int_{B_{r}} |\Delta^2 u|^2 \varphi^2 dV + \frac{C}{(r-s)^2}\int_{B_{r}} |\Delta u|^2 dV.
\end{equation*}

\end{proof}

The following Lemma allows to set up the induction
argument in the proof of our main result. The proof
is by induction, the base case being the biharmonic case
(Lemma \ref{cor:L2-Estimate Laplacian Biharmonic}).

\begin{lemma}\label{thm:L2-Estimate Laplacian}
Let $(M,g)$ be an $n$–dimensional complete Riemannian manifold with $\mathrm{Ric}_g\ge 0$.
Fix $a>0$ and let $u\in C^\infty(B_{r})$ be k-polyharmonic, i.e.\ $\Delta^k u=0$ on $B_{r}$, for $k\geq 2$.
Then, there exists a constant $C>0$
such that for every $s,r$, with $0<\frac{r}{8}<s < r$,
\begin{equation}\label{eq:L2-HF-polh}
\int_{B_s}|\Delta u|^2 dV \leq \frac{C}{(r-s)^4}\int_{B_{r}} u^2 dV.
\end{equation}
\end{lemma}

\begin{proof}
We proceed by induction. For $k=2$, the estimate is just \ref{cor:L2-Estimate Laplacian Biharmonic}. 
Assume that the estimate holds for $(k-1)$-polyharmonic functions, with $k\geq 3$.\\

Fix $0<s< r$. By Lemma \ref{corollary:special_cutoff} we can choose $\varphi\in C_c^\infty(B_{r}(p))$ such that
\[
0\le \varphi\le 1,\qquad
\varphi\equiv 1 \;\text{ on } B_s,\qquad
\mathrm{supp}\,\varphi \subset B_{r}.
\]
with
\begin{equation}\label{eq:cutoff-grad-lap2}
|\nabla\varphi|^2\le \frac{C}{(r-s)^2}\varphi,
\qquad
|\Delta\varphi|\le \frac{C}{(r-s)^2},
\end{equation}
and hence 
\begin{equation}
\label{eq:lap22}
    \frac{|\Delta(\varphi^2)|}{\varphi}\leq \frac{2\varphi |\Delta \varphi|}{\varphi} +\frac{2 |\nabla \varphi|^2}{\varphi}\leq \frac{C}{(r-s)^2}.
\end{equation}

By Young's inequality, for $\eta >0$, $\epsilon >0$ and $\delta >0$, we have
\begin{align*}
    \int_{B_{r}} |\Delta u|^2 \varphi^2 dV &= \int_{B_{r}} u \Delta \left(\varphi^2 \Delta u\right) dV,\\
    &= \int_{B_{r}} u \Delta^2 u \varphi^2 dV + \int_{B_{r}} u \Delta u \Delta \varphi^2 dV + 4 \int_{B_{r}} u \varphi \langle \nabla \varphi, \nabla \Delta u\rangle dV\\
    &\leq \frac{\eta}{2}\int_{B_{r}} u^2 \varphi^2 dV + \frac{1}{2 \eta} \int_{B_{r}} |\Delta^2 u|^2 \varphi^2 dV+\int_{B_{r}} \frac{\epsilon}{2} u^2 dV\\
    &\quad + \int_{B_{r}} |\Delta u|^2 \frac{|\Delta \varphi^2|^2}{2 \epsilon} dV +4 \int_{B_{r}} \frac{1}{2 \delta} u^2 \varphi^2 dV + 4 \int_{B_{r}} \frac{\delta}{2} |\nabla \Delta u|^2 |\nabla \varphi|^2 dV.
\end{align*}
Taking $\epsilon = \dfrac{2 C}{(r-s)^4}$ yields
\begin{align*}
    \int_{B_{r}} |\Delta u|^2 \varphi^2 dV & \leq \frac{\eta}{2}\int_{B_{r}} u^2 \varphi^2 dV + \frac{1}{2 \eta} \int_{B_{r}} |\Delta^2 u|^2 \varphi^2 dV+\frac{C}{(r-s)^4}\int_{B_{r}}  u^2 dV\\
    &\quad + \frac{(r-s)^4}{4C}\int_{B_{r}} |\Delta u|^2 |\Delta \varphi^2|^2 dV + \int_{B_{r}} \frac{2}{\delta}  u^2 \varphi^2 dV\\
    & \quad +  \int_{B_{r}} 2 \delta |\nabla \Delta u|^2 |\nabla \varphi|^2 dV.
\end{align*}
Applying the bounds on the cutoff function given in \eqref{eq:cutoff-grad-lap2} and \eqref{eq:lap22}
\begin{align*}
    \int_{B_{r}} |\Delta u|^2 \varphi^2 dV & \leq \frac{\eta}{2}\int_{B_{r}} u^2  dV + \frac{1}{2 \eta} \int_{B_{r}} |\Delta^2 u|^2 \varphi^2  dV+\frac{C}{(r-s)^4}\int_{B_{r}}  u^2 dV\\
    &\quad + \frac{(r-s)^4}{4C} \frac{C}{(r-s)^4}\int_{B_{r}} |\Delta u|^2 \varphi^2  dV + \int_{B_{r}} \frac{2}{\delta}  u^2 \varphi^2  dV\\
    & \quad + \frac{C}{(r-s)^2} \int_{B_{r}} 2 \delta  |\nabla \Delta u|^2 dV.
\end{align*}

Taking $\delta = \dfrac{(r-s)^4}{32C^2} \varphi^2$, 

\begin{align*}
    \int_{B_{r}} |\Delta u|^2 \varphi^2 dV & \leq \frac{\eta}{2}\int_{B_{r}} u^2  dV + \frac{1}{2 \eta} \int_{B_{r}} |\Delta^2 u|^2 \varphi^2  dV+\frac{C}{(r-s)^4}\int_{B_{r}}  u^2 dV\\
    &\quad + \frac{1}{4}\int_{B_{r}} |\Delta u|^2 \varphi^2  dV +   \frac{64C^2}{(r-s)^4}\int_{B_{r}}  u^2   dV\\
    & \quad + \frac{1}{16} \frac{(r-s)^2}{C} \int_{B_{r}}   |\nabla \Delta u|^2 \varphi^2 dV.
\end{align*}

By Lemma \ref{lem:L2-Estimate_Gradient_N} 
we have that
\begin{align*}
    \int_{B_{r}} |\Delta u|^2 \varphi^2 dV & \leq \frac{\eta}{2}\int_{B_{r}} u^2  dV + \frac{1}{2 \eta} \int_{B_{r}} |\Delta^2 u|^2 \varphi^2  dV+\frac{C}{(r-s)^4}\int_{B_{r}}  u^2 dV\\
    &\quad + \frac{1}{4}\int_{B_{r}} |\Delta u|^2  dV
     +\frac{64C^2}{(r-s)^4} \int_{B_{r}}   u^2   dV \\
     &+  \frac{1}{16} \frac{(r-s)^4}{C^2} \int_{B_{r}}  |\Delta^2 u|^2 \varphi^2   dV
      \quad + \frac{1}{16}\int_{B_{r} }  |\Delta u|^2 dV.
\end{align*}

As $u$ is $k-$polyharmonic then $\Delta u$ is $(k-1)$-polyharmonic and thus, by the induction hypothesis, 

\begin{equation}
\label{eq:lap2}
    \int_{B_{r}} |\Delta^2 u|^2 \varphi^2 dV \leq \frac{C^2}{(r-s)^4} \int_{B_{r}} |\Delta u|^2 dV.
\end{equation}

Therefore, using \eqref{eq:lap2} in the previous estimate,
\begin{align*}
    \int_{B_s} |\Delta u|^2 dV &\leq \left[\frac{\eta}{2}+\frac{C}{(r-s)^4}+\frac{64C^2}{(r-s)^4}\right] \int_{B_{r}} u^2 dV\\
    &\quad + \left[\frac{1}{2 \eta} \frac{C^2}{(r-s)^4}+\frac{1}{4}+\frac{1}{16}+\frac{1}{16}\right]\int_{B_{r}} |\Delta u|^2 dV.
\end{align*}

Taking $\eta = \dfrac{4 C^2}{(r-s)^4}$ gives

\begin{equation}
\label{eq:lap223}
    \int_{B_s} |\Delta u|^2 dV \leq \frac{C_n}{(r-s)^4} \int_{B_{r}} u^2 dV + \frac{1}{2}\int_{B_{r}} |\Delta u|^2 dV.
\end{equation}

Applying the hole-filling lemma (Lemma \ref{lemma:hole-filling}) to \eqref{eq:lap223} with $\alpha=4$, we have that for $0<\rho\le r$,

\begin{equation*}
\label{est:lap pol}
    \int_{B_s}|\Delta u|^2 dV \leq \frac{C}{(r-s)^4}\int_{B_{r}} u^2 dV.
\end{equation*}
    
\end{proof}

We have the following important consequence which follows by induction from the previous lemma.

\begin{corollary}\label{cor:L2-Estimate Laplacian k-1}
Let $(M,g)$ be an $n$–dimensional complete Riemannian manifold with $\mathrm{Ric}_g\ge 0$.
Fix $a>0$ and let $u\in C^\infty(B_{r})$ be k-polyharmonic, i.e.\ $\Delta^k u=0$ on $B_{r}$, for $k\geq 2$.
Then, there is a constant $C>0$ such that
for every $r>0$ it holds that
\begin{equation*}
\label{eq:L2-HF-polhk-1}
\int_{B_r}|\Delta^{k-1} u|^2 dV \leq 
\frac{2^{4\left(k-2\right)}C}{r^{4(k-1)}}\int_{B_{2^{k-1}r}} u^2 dV.
\end{equation*}
\end{corollary}

\subsection{Proof of Main Theorem  \ref{thm:main}}
\label{sect:proof_main_thm}
We are ready to prove our main result. Let $u$ be a $k$-polyharmonic function.
By Corollary \ref{cor:L2-Estimate Laplacian k-1} we can estimate
\begin{equation*}
\label{ineq:Cacciopoli_laplacian}
\int_{B_{r}}|\Delta^{k-1} u|^2\,dV
\leq \frac{C'}{r^{4(k-1)}}\int_{B_{2^{k-1}r}}u^2\,dV.
\end{equation*}
By the volume doubling property (which holds if $\mbox{Ric}\geq 0$), we have
\begin{eqnarray*}
\frac{1}{\mbox{Vol}\left(B_{r}\right)}\int_{B_{r}}|\Delta^{k-1} u|^2\,dV
&\leq& \frac{C'}{r^{4(k-1)}}\frac{1}{\mbox{Vol}\left(B_{2^{k-1}r}\right)}\int_{B_{2^{k-1}r}}u^2\,dV,
\end{eqnarray*}
and hence, since $\Delta^{k-1} u$ is harmonic, then $|\Delta^{k-1} u|$ is subharmonic, by the mean value inequality (Lemma \ref{lemma:mean_value_ineq}),
\begin{eqnarray*}
\sup_{x\in B_{r/2}}\left|\Delta^{k-1} u\right|^2 &\leq& 
\frac{C''}{r^{4(k-1)}}\frac{1}{\mbox{Vol}\left(B_{2^{k-1}r}\right)}\int_{B_{2^{k-1}r}}u^2\,dV\\
&\leq& \frac{C''}{r^{4(k-1)}}\sup_{x\in B_r\left(p\right)}\left|u\left(x\right)\right|^2
=\frac{1}{r^{4(k-1)}}o\left(r^{4(k-1)}\right)
\end{eqnarray*}
and hence if $u=o(r^{2(k-1)})$, letting $r\rightarrow \infty$, it 
follows that
\[
\Delta^{k-1} u \equiv 0,
\]
i.e., $u$ is $(k-1)$-polyharmonic. For the last statement, let $u$ be a polyharmonic function of
sublinear growth, i.e.
\[
\sup_{B_r}|u| = o(r) \quad \text{as } r\to\infty.
\]
For every $k\ge2$ we have
\[
\frac{\sup_{B_r}|u|}{r^{2(k-1)}}
= o\big(r^{1-2(k-1)}\big) \;\longrightarrow\;0.
\]
In particular $u = o\big(r^{2(k-1)}\big)$ for each $k\ge2$.
Assume that $u$ is $k$–polyharmonic for some $k\ge2$. By the first
part of the theorem, $u$ is then $(k-1)$–polyharmonic. Iterating
this argument for $k,k-1,\dots,2$, we conclude that $u$ is harmonic
and it is sublinear growth.

\medskip
Finally, by Yau’s Liouville theorem for harmonic functions on
complete manifolds with $\mathrm{Ric}_g\ge0$, any harmonic function
with sublinear growth must be constant. Hence $u$ is constant on
$M$, which is what we wanted to prove.
%\begin{proof}[Idea of the proof]
%Since $u\ge0$ and $\Delta u\ge0$, one has
%$\Delta(u^{2})=2u\,\Delta u + 2|\nabla u|^{2}\ge0$, so $u^{2}$ is subharmonic on $B_{2R}$.
%Apply the $L^{1}$ mean value inequality for nonnegative subharmonic functions
%to $u^{2}$ on $B_{2R}$ to obtain \eqref{eq:MV-u2}; the factor
%$\exp(C\sqrt{K}\,R)$ is the usual Ricci-lower-bound contribution.
%\end{proof}

\appendix
\section{The biharmonic case}

As stated in the introduction, inequality \eqref{eq:Cacciopoli_Laplacian} 
does not appear in \cite{BravoCortissoz2025}. However, 
the estimate can be inferred from the proofs of Lemmas 5 and 6
in \cite{BravoCortissoz2025} as follows. From the proof of the local 
energy estimate for the Laplacian, Lemma 5, for a cutoff function 
$\varphi\in C_c^\infty(B_{2r})$ satisfying for $0<\frac{1}{8}r<s<r$
\[
\varphi\equiv1 \text{ on }B_s,\quad 0\le\varphi\le1,\quad
|\nabla\varphi|\le \frac{C_0}{r-s}\varphi,\quad
|\Delta\varphi|\le \frac{C_0}{(r-s)^2},
\]
we have
\begin{lemma}
\label{lem:laplacian_estimate}
Let $(M,g)$ be an $n$–dimensional complete Riemannian manifold  with $\mathrm{Ric}_g\ge 0$. Let $B_{\rho}\colon = B_{\rho}\left(p\right)$.
Fix $a>0$ and let $u\in C^\infty(B_{a})$ be biharmonic, i.e.\ $\Delta^2 u=0$ on $B_{2a}$. Then there exists a constant $C=C(n)>0$ such that, for every pair of radii $0<\dfrac{1}{8}r<s<r\le a$,
\[
\int_{B_s}|\Delta u|^2\,dV
\leq \frac{C}{(r-s)^2}\int_{B_r}|\nabla u|^2\varphi^2\,dV
\;+\; \frac12\int_{B_r} |\Delta u|^2\,dV.
\]
\end{lemma}

From this lemma, using a hole-filling technique follows 
that for every $0<\dfrac{1}{8}r<s<r<\infty$ 
\[
\int_{B_s}|\Delta u|^2\,dV
\leq \frac{C}{(r-s)^2}\int_{B_r}|\nabla u|^2\varphi^2\,dV.
\]

From the proof of Lemma 6, the Cacciopoli inequality
for a biharmonic function, for a cutoff 
$\varphi\in C_c^\infty(B_{2r})$ such that $0\leq \varphi\leq 1$,
$\varphi\equiv 1$ on $B_s$, and for $\dfrac{1}{8}r<s<r$, it also
satisfies
\[
\left|\nabla \varphi\right|^2\leq \frac{C}{\left(r-s\right)^2}\varphi\quad
    \text{ and } \quad
    \frac{\left|\Delta\left(\varphi^4\right)\right|^2}{\varphi^4}\leq \frac{C}{\left(r-s\right)^4},
\]
we have
\begin{lemma}
\label{lem:Caccio-1-biharmonic-Br}
Let $(M,g)$ be a smooth Riemannian manifold  
such that $\mbox{Ric}\geq 0$.
Let $0<\dfrac{1}{8}r<s<r<\infty$
and let $B_\rho:=B_\rho(p)$. If $u\in C^\infty(B_{2r})$ is \emph{biharmonic}, i.e.\ $\Delta^2 u=0$ in $B_{2r}$, then there exists a constant $C=C(n)>0$ such that
\begin{equation*}\label{eq:Caccio-int-Br}
\int_{B_r} |\nabla u|^2\varphi^2\,dV \;\le\; \frac{C}{(r-s)^2}\int_{B_r} u^2\,dV.
\end{equation*}
\end{lemma}

Putting all these together, \eqref{eq:Cacciopoli_Laplacian} follows readily.

%\bibitem{LT}
%L. ~Ni and L.~F. ~Tam, 
%\textquotedblleft Plurisubharmonic functions and the structure of complete K\"ahler manifolds with nonnegative curvature\textquotedblright,
% J. Differential Geom., \textbf{64} (3), 457--524 (2003).

%\bibitem{Neel}
%R.~Neel, \textquotedblleft
%Brownian Motion and the Dirichlet Problem at Infinity
%on Two-dimensional Cartan-Hadamard Manifolds \textquotedblright, Potential Anal. \textbf{41},
%443--462 (2014).

%\bibitem{Neel21}
%R.~ Neel, \emph{
%Geometric and Martin boundaries of a Cartan-Hadamard surface
%},
%ALEA Lat. Am. J. Probab. Math. Stat. \textbf{18} no. 2 (2021), 1669--1687.

%\bibitem{Peetre}
% J. Peetre, \emph{Absolute convergence of eigenfunction expansions}, Math. Ann. 169 (1967), 
% 307--314. 

%\bibitem{Saloff-Coste92}
%L. Saloff-Coste, 
%\emph{A note on Poincaré, Sobolev, and Harnack inequalities}, Int. Math. Res. Not. 2, 27--38
%(1992).

%\bibitem{Sullivan}
%D. ~Sullivan,\emph{The Dirichlet problem at infinity for a negatively curved manifold},
%J. Differential Geom. \textbf{18}, no. 4, 723--732 (1983).

%\bibitem{Taylor}
%M. E. Taylor, \textquotedblleft
%Fourier series on compact Lie groups \textquotedblright,
%Proc. Amer. Math. Soc. \textbf{19} (1968), 1103--1105. 

%\bibitem{Y}
%S. ~T. ~Yau, \textquotedblleft Harmonic Functions on Complete Riemannian Manifolds
%\textquotedblright, Comm. Pure Appl. Math. \textbf{28} 201--228 (1975).
% J. Differential Geom., \textbf{64} (3), 457--524 (2003).

\end{document}